\documentclass{siamart190516}

\usepackage{amsmath} 
\usepackage{amssymb}  
\usepackage{amsfonts}
\usepackage{graphicx}
\usepackage{epstopdf}
\usepackage{epsfig} 

\newcommand{\TheTitle}{Generalization of Affine Feedback Stock  \\Trading Results to Include Stop-Loss Orders}
\newcommand{\TheAuthors}{Chung-Han Hsieh}

\headers{\TheTitle}{\TheAuthors}

\title{{\TheTitle}
	\thanks{
		The author would like to thank his dissertation advisor Professor B. Ross Barmish at the Boston University for his comments on the early drafts of this paper.
}}

\author{
	Chung-Han Hsieh\thanks{Department of Electrical and Computer Engineering, University of Wisconsin--Madison,
		(\mbox{\email{chunghan.hsieh@wisc.edu}}).}
}

\begin{document}
	
	\maketitle
	
	\begin{abstract}
		The takeoff point of this paper is to generalize the existing stock trading results for a class of affine feedback controller to include consideration of a stop-loss order. 
		Using the geometric Brownian motion as the underlying stock price model, our main result is to provide a closed-form expression for the cumulative distribution function for the trading profit or loss. 
		In addition, we show that the affine feedback controller with stop-loss order indeed generalizes the result without stop order in the sense of distribution function. 
		Some simulations and illustrative examples are also provided as supporting evidence of the theory.
		Moreover, we provide some technical results aimed at addressing the issues about \textit{survivability}, \textit{cash-financing} considerations, \textit{long-only} property, and lower bound of the expected gain or loss.
	\end{abstract}
	
	
	
	\vspace{5mm}
	\section{Introduction}
	This paper is along an emerging line of researches which studies stock trading using control-theoretic approaches; e.g., see~\cite{Barmish_Primbs_2011,Barmish_Primbs_2016,hsieh2017drawdown,Hsieh_Barmish_2017_domination,calafiore2009affine,zhang2001stock,hsieh2019positive,hsieh2019impact}. 
	The takeoff point of this paper is to generalize the results for a class of affine feedback controller to include the consideration of stop-loss orders. 
	In particular, a \textit{stop-loss} order is a type of ``price-contingent" order described as follows: The stock is to be sold if the stock price falls below a prespecified price called \textit{stop price}. 
	When the stop price is reached, a stop order becomes a market order and get executed at the price. 
	For a general discussion on other typical price-contingent orders in stock trading, we refer the reader to \cite{bodie2009investments} and \cite{luenberger1997investment}.
	In the sequel, instead of calling the stop-loss order, we may sometimes call it as a stop order for simplicity.
	
	The main questions we would like to address in the paper is as follows: Given a geometric Brownian motion (GBM) as the underlying stock price model, what is the cumulative distribution function (CDF) for the trading profit or loss, call it $g(t)$, when the affine control with stop-loss order is considered? 
	Moreover, is it possible to obtain a closed-form characterization for the desired CDF for $g(t)$ so that a theoretician or practical trader might be benefited from it?   
	To the best of our knowledge, there is no literature discussing this issue and we view that this paper indeed plays the role for filling a void on this topic. 
	To be more specific, with a GBM price traded by an affine control with stop order, our main contribution is to provide a complete characterization on the CDF for $g(t)$ in closed form. 
	
	To close this brief introduction, we mention some related work regarding the stock trading which also includes stop order considerations; e.g.,  see \cite{muller1997stop} for a studies on risk measure in terms of the stochastic \textit{stop-loss order}, \cite{zhang2001stock} for a research on optimal selling rule via standard control-theoretic approach, and our prior work, \cite{hsieh2017drawdown,Hsieh_Barmish_2017_domination}, for a demonstration that a stop-loss trading behavior is also seen within our drawdown modulated feedback framework. 
	In addition,  some empirical trading tests via \textit{limit orders}, which is closely related to stop-loss order, are studied in  \cite{handa1996limit}. It indicates that the trading with limit orders indeed dominates the trading via market orders.  
	
	
	\subsection{Plan For the Remainder of The Paper} 
	In Section~\ref{SEC: Problem Formulation}, we state the problem formulation. Specifically, we use the geometric Brownian motion as the stock price model and our aim is to study the trading performance for an affine feedback control law with stop-loss order. 
	Subsequently, in Section~\ref{Sec: Preliminary Notations}, we introduce some preliminaries and shorthand notations  which are frequently used in the rest of the paper. 
	In Sections~\ref{SEC: CDF for buy and hold}--\ref{SEC: CDF for Cash-Financing}, we provide our main results, which characterize the CDF for the cumulative trading or loss functions for three cases; i.e., buy and hold ($K=1$), ``bold" investment ($K>1$), and ``timid" investment ($0<K<1$). 
	The corresponding remarks and illustrative examples are also provided.
	Then, in Section~\ref{SEC:Some Technical Results}, we also give some additional technical results regarding the issues of survivability, control input properties, long-only property, and  lower bound of expected trading gain-loss. 
	Finally, in Section~\ref{SEC:conclusions}, some concluding remarks and future work are discussed.
	
	\vspace{5mm}
	\section{Problem Formulation}\label{SEC: Problem Formulation}
	In the sequel, we assume the stock trading occurs within an ``idealized market." That is, we assume zero transaction costs, zero interest rates, perfect liquidity, and continuous trading. These assumptions arise in the finance literature in the context of ``firctionless" market; e.g., see~\cite{Barmish_Primbs_2016} and \cite{merton1992continuous} and provide a ``benchmark" setting when one wants to apply the theory developed in this paper in practice.
	
	To study the stock trading performance, it is convenient to establish a mathematical model which can be used to characterize the underlying stock price movement. 
	In mathematical finance, the geometric Brown motion (GBM) is often used as a standard price model, which is characterized by the two parameters: \textit{drift} and \textit{volatility}; see~\cite{Barmish_Primbs_2011,Barmish_Primbs_2016,karatzas1998methods, campolieti2016financial,luenberger1997investment}. 
	
	\subsection{GBM Price Model and Trader's Account Value}
	Consistent with the literature, for $t\geq 0$, we consider a GBM model for stock prices~$S(t)$ described by the stochastic differential equation
	$$
	\frac{dS}{S}=\mu dt + \sigma  dW; \ S(0) = S_0
	$$
	where $\mu$ is  drift and $\sigma > 0$ is   volatility and $W$ is a standard Wiener process resulting from $dW$; see~\cite{luenberger1997investment} for further detail.
	The drift parameter $\mu$ captures the annualized expected return, and the parameter $\sigma$ represents the annualized standard deviation corresponding to the underlying process.
	
	Having defined the stock price model, we now focus on the relationship between trader's account value and trading profit or loss. In particular, changes in the trader's account value correspond to changes in the profit or loss level. 
	Thus, for $t \geq 0,$ with initial account value $V(0):=V_0>0$, we have
	\[
	V(t) = V_0 + g(t)
	\]
	which implies $g(0)=0$.
	In the sequel, as mentioned previously, the quantity $g(t)$  denotes the \textit{cumulative profit or loss} over the time interval~$[0,t]$.

	\subsection{Affine Feedback Controller with Stop-Loss Order}
	When the price movement is against the trader, reducing the trading loss is crucial to preserve its capital. In practice, a typical level for a stop-loss order to be triggered is around $10\%$--$30\%$ for ``active" traders and~$8\%$ for ``conservative" traders; see \cite{o1988make}. 
	
	To include the stop-loss order consideration into our formulation, we proceeds as follows:  Given a constant $S_* >0$ with $S_* < S_0$ which we call the {\it stop price},  the trade will be terminated if 
	$$
	S(t^*) = S_*
	$$
	for some time $t^* \ge 0$. 
	This being the case, we obtain the \textit{stopped GBM}, call it~$\widetilde{S}(t)$, which is described as follows:
	\begin{align*}
	\widetilde{S}(t) := 
	\begin{cases}
	S(t), & t < t^*;\\
	S_*, & t \ge t^*.
	\end{cases}
	\end{align*}
	We note here that $\widetilde{S}(t) \geq S_*$ for all $t \geq 0.$
	Now, letting $u(t)$  be the control which representing the investment level along sample path with $\widetilde{S}(t)$, for $t \ge t^*$, the trade is stopped and we take $u(t) := 0$. For times $t$ such that the trade is \textit{not} stopped, we consider the control $u(t)$ determined by the form of affine feedback
	$$ \label{eq: control}
	u(t) := u_0 + K g(t)
	$$
	with $u_0:=u(0)>0$ and $K> 0$.\footnote{It is interesting to note that if $u_0:=KV(0)$, then, via a straightforward calculation, the affine feedback law reduces to the classical linear feedback; i.e., $u(t) = KV(t)$. In Section~\ref{SEC:Some Technical Results}, we see that the choice of $u_0$ may affect ``survivability" (no-bankruptcy). 
	} 
	This strategy is said to be \textit{going long} which represents the case where the trade is profitable with rising prices and is losing with falling prices; see also Lemma~\ref{lemma: long only property} in Section~\ref{SEC:Some Technical Results}.
	For $K=1$, the affine feedback controller above reduces to {\it buy and hold} trading strategy; see \cite{Barmish_Primbs_2016} for further detail.
	
	\subsection{Problem Statement}
	Given a GBM stock price, we consider a class of an affine feedback control described in equation~(\ref{eq: control}) and stop-loss order with stop price $S_*<S_0$. Our main goal is to seek a closed-form expression of CDF for cumulative trading profit or loss $g(t)$; i.e., for any $x \in \mathbb{R}$,
	\[
	F(x,t) :=P(g(t)\leq x).
	\]

	
	\vspace{5mm}
	\section{Preliminaries and Shorthand Notations}\label{Sec: Preliminary Notations} If the stop-loss order is not involved in the formulation, then the GBM price model leads to a stochastic differential equation which describes the incremental change in the profit or loss given by
	\[
	dg = \frac{dS}{S}u.
	\]
	The following lemma, which was initially proved in \cite{Barmish_Primbs_2016}, characterizes  the cumulative profit or loss function without stop order.
	
	\begin{lemma}[Barmish and Primbs \cite{Barmish_Primbs_2016}]\label{lemma: Barmish and Primbs g(t)}
		For $t\geq 0$, the affine feedback controller $u(t)=u_0+Kg(t)$ with $u_0>0$ and $K>0$ leads to the cumulative profit or loss
		\[
		g(t) = \frac{u_0}{K} \left( \left( \frac{S(t)}{S_0} \right)^K e^{{\frac{1}{2}\sigma^2 (K-K^2)t}} - 1\right).
		\]
	\end{lemma}
	
	\begin{proof} Omitted; see~\cite[Theorem 6.2]{Barmish_Primbs_2016}.
	\end{proof}
	
	\subsection{Remark} When there is no stop-loss order involved, one may readily verify that $g(t) \geq {-u_0}/{K}$ for all~$t\geq 0$ with probability one. 
	However, as seen later in this paper, when the trading involves stop order, the lower bounds for cumulative profit or loss $g(t)$ may be different. Indeed, we remind the reader that the expression for $g(t)$ above and its modifications, such as $g_*$ and $g_*(t)$ defined below, are frequently seen in the rest of the paper. 
	
	\subsection{Shorthand Notations}
	In the sequel, to simplify the exposition, we take
	$$
	Z_* :=  \left(\frac{S_*}{S_0} \right)^{\frac{{2\mu }}{{{\sigma ^2}}} - 1}; \; g_* := u_0 \left( \frac{S_*}{S_0} - 1 \right),
	$$
	and, for  $t>0$ we define 
	\begin{align}
	&g_*(t) := \frac{u_0}{K}\left( \left( \frac{S_*}{S_0} \right)^K {e^{\frac{1}{2}{\sigma ^2}(K - K^2)t}} - 1 \right)\label{eq:g_star_t}.
	\end{align}
	We note here that $g_*(t)$ is identical to the $g(t)$ if $S(t)$ is replaced by the stop price~$S_*$ and
	$g^*$ can be obtained from~$g_*(t)$ by simply setting~$K=1$.
	For $z,t>0$, we also define
	\begin{align}\label{eq: X_z_t}
	X(z,t) := \frac{ \log{z} + \left( \mu  - \frac{1}{2}\sigma ^2 \right)t}{\sigma \sqrt t }.
	\end{align}
	In addition, for $t>0$ and $x > -u_0/K$ with $K\neq 1$, we let
	\begin{align}
	&A \left( x \right) := \frac{-2}{{\sigma ^2}K(K - 1)}\log \left[ \left( {\frac{S_0}{S_*}} \right)^K\left( \frac{K}{u_0}x + 1 \right) \right];\label{eq: A_x}\\[.5ex]
	& B(x,t) := {{{\left[ {\left( {\frac{K}{{{u_0}}}x + 1} \right){e^{ - \frac{1}{2}{\sigma ^2}(K - {K^2})t}}} \right]}^{1/K}}}. \label{eq: B_x_t}
	\end{align}
	We should note that $A(x)$ above is well-defined within the range of $x$. To see this, one can simply verify that the argument of the logarithmic function in equation~(\ref{eq: A_x}) is strictly positive. That is,
	\[
	\left( {\frac{S_0}{S_*}} \right)^K\left( \frac{K}{u_0}x + 1 \right)>0.
	\]
	
	\subsection{Preliminary Results}
	In the sequel, we use $F_0(x,t)$ to denote the CDF of the trading profit or loss without stop-loss order. Let $\Phi(\cdot)$  be the cumulative distribution function (CDF) of the standard normal distribution satisfying
	$$
	\Phi(x):=\frac{1}{\sqrt{2\pi}}\int_{-\infty}^x e^{-t^2/2}dt,
	$$
	
	the following lemma summarizes the CDF of $g(t)$ when there is no stop-loss order involved. 

	\begin{lemma}[CDF for $g(t)$ without Stop Order]\label{lemma: CDF for g(t) without stop order}
		For $t>0$, consider a GBM stock price with drift $\mu$ and volatility~$\sigma$, an affine feedback controller with {gain $K>0$}.  Then the CDF for cumulative trading profit or loss, call it  $F_0(x,t) $, is as follows: 
		\begin{align} 
		F_0(x,t):=\begin{cases}
		0 & \text{ if } x \le -u_0/K;\\[.5ex]
		1  - \Phi \left( X\left( \frac{1}{B(x,t)},t \right) \right) & \text{ if } x > -u_0/K 
		\end{cases}
		\end{align}
		
	\end{lemma}
	
	\begin{proof}
		Recalling the gain-loss function $g(t)$ from Lemma~\ref{lemma: Barmish and Primbs g(t)},
		we write
		\begin{align*}
		F_0(x,t) 
		&= P(g(t) \le x)\\[.5ex]
		&= P\left( \frac{u_0}{K}\left[ \left( {\frac{S( t )}{S_0}} \right)^Ke^{\frac{1}{2}{\sigma ^2}\left( K - K^2 \right)t} - 1 \right] \le x \right).
		\end{align*} 
		Note that $g(t) \geq -u_0/K$ for all $t$. Thus, for $x \le -u_0/K$, the CDF
		$
		F_0(x,t)  = 0. 
		$
		On the other hand, for $x > -u_0/K$, the CDF becomes
		\begin{align*}
		F_0(x,t) 
		&= P\left( \log  \frac{S(t)}{S_0} \le \frac{1}{K}\log \left[ \left( \frac{K}{u_0}x + 1 \right){e^{ - \frac{1}{2}{\sigma ^2}\left( {K - {K^2}} \right)t}} \right] \right)\\[.5ex]
		&= P\left( \log  \frac{S(t)}{S_0} \le B(x,t) \right)
		\end{align*}
		where the last equality holds by applying the definition of $B(x,t)$, which is defined in equation~(\ref{eq: B_x_t}).
		Now, since ${S(t)}/{S_0}$ is log-normal distributed; i.e.,
		\[
		\log \frac{S(t)}{S_0} \sim \mathcal{N}\left ( \left ( \mu - \frac{1}{2} \sigma^2\right)t, \sigma^2t \right )
		\]
		where $\mathcal{N}\left ( \left ( \mu - \frac{1}{2} \sigma^2\right)t, \sigma^2t \right )$ means the normal distribution with mean $(\mu -  \sigma^2/2  )t$ and variance $\sigma^2 t$. 
		Thus, the CDF can be written explicitly as follows:
		\begin{align*}
		F_0(x,t) 
		&= \Phi \left( \frac{ \log B(x,t) - \left( \mu  - \frac{1}{2}{\sigma ^2} \right)t}{\sigma \sqrt t } \right)\\[.5ex]
		&= 1-\Phi \left( X\left(\frac{1}{B(x,t)},t\right) \right)
		\end{align*}
		where the last equality holds by applying the facts that $B(x,t)>0$ for $x>-u_0/K$ and $t>0$ and $\Phi(-z)=1-\Phi(z)$ for all $z \in \mathbb{R}$, and the definition of $X(z,t)$, which is defined in equation~(\ref{eq: X_z_t}). Hence,
		the proof is complete.
	\end{proof}
	
	Let 
	$
	t^* := \inf \{t \geq 0: S(t) \le S_* \},
	$ 
	which is a stopping time.
	The next simple lemma generalizes Lemma~\ref{lemma: Barmish and Primbs g(t)} to the case where the affine control with stop-loss order is considered. 
	
	\begin{lemma}\label{lemma: g_t for stopped GBM}
		For $t \geq 0,$ consider an affine feedback controller with gain $K>0$ and a stop price $S_*<S_0.$ Then we have
		\[
		g(t) = \frac{u_0}{K}\left( \left( \frac{\widetilde{S}(t)}{S_0} \right)^K {e^{\frac{1}{2}{\sigma ^2}(K - K^2)t}} - 1 \right).
		\]
	\end{lemma}
	
	\begin{proof}
		For $t \geq t^*$, it follows that $S(t) = S_*.$ Thus, Lemma~\ref{lemma: Barmish and Primbs g(t)} tells us that $$g(t) = \frac{u_0}{K}\left( \left( \frac{S_*}{S_0} \right)^K {e^{\frac{1}{2}{\sigma ^2}(K - K^2)t}} - 1 \right).$$
		On the other hand, for $t < t^*,$ 
		the  trading gain $g(t)$ can be found in terms of $g(t^*)$ and the remaining gain/loss from time $t$ to~$t^*$, call it  $g_r(t, t^*)$;~ namely,
		\[
		g({t^*}) = g(t) + g_r(t,t^*).
		\]
		Thus, we can write
		\begin{align*}
		g(t) &= g({t^*}) -  g_r(t,t^*)\\[.5ex]
		&  = \frac{{{u_0}}}{K}\left[ {{{\left( {\frac{{{S_*}}}{{{S_0}}}} \right)}^K}{e^{\frac{1}{2}	{\sigma ^2}(K - {K^2}){t^*}}} - 1} \right] \\[.5ex]
		& \ \ \ \ \ \  - \frac{{{u_0 + K g(t)}}}{K}\left[ {{{\left( {\frac{{S_*}}{{{S(t)}}}} \right)}^K}{e^{\frac{1}{2}{\sigma ^2}(K - {K^2})\left( {{t^*} - t} \right)}} - 1} \right]\\[.5ex]
		& = \frac{{{u_0}}}{K}\left[ {{{\left( {\frac{{S(t)}}{{{S_0}}}} \right)}^K}{e^{\frac{1}{2}{\sigma ^2}(K - {K^2})t}}\; - 1} \right].
		\end{align*} 
		Combining the two cases, we obtain 
		\[
		g(t) = \begin{cases}
		\frac{u_0}{K}\left( \left( \frac{S(t)}{S_0} \right)^K {e^{\frac{1}{2}{\sigma ^2}(K - K^2)t}} - 1 \right), & \text{ if } t < t^*;\\
		\frac{u_0}{K}\left( \left( \frac{S_*}{S_0} \right)^K {e^{\frac{1}{2}{\sigma ^2}(K - K^2)t}} - 1 \right), & \text{ if } t \geq t^*.
		\end{cases}
		\]
		which is equivalent to the  formula stated in the lemma.
	\end{proof}

	In the next sections to follow, we provide our main result by extending the result obtained in Lemma~\ref{lemma: CDF for g(t) without stop order} to include stop-loss order. 
	To be more specific, we provide closed-form characterizations of the CDF for the trading profit or loss $g(t)$ in three different cases: The buy and hold case, which corresponds to $K=1$, is analyzed in Section~\ref{SEC: CDF for buy and hold}. 
	Subsequently, the ``bold investment" case, which corresponds to $K > 1$, is addressed in Section~\ref{SEC: CDF for Leverage Case}. 
	Then Section~\ref{SEC: CDF for Cash-Financing} addresses the case for $0<K<1$, which corresponds to the ``timid investment."\footnote{In finance, the timid investment is closely related to the so-called \textit{cash-financing} condition. That is, the trader's investment can not exceed the account value; i.e., $|u(t)| \leq V(t)$ for all $t\geq 0$; see also Remark~\ref{remark: cash-financing}.}
	
	\vspace{5mm}
	\section{Buy and Hold Case (\texorpdfstring{$K=1$}{K=1})}\label{SEC: CDF for buy and hold}
	As mentioned in Section~\ref{Sec: Preliminary Notations}, the CDF for the cumulative trading profit or loss function is characterized for the $K =1$ case. 
	In this setting, the corresponding feedback control law becomes $u(t)=u_0 + g(t)$.
	
	\begin{theorem}\label{theorem1}
		For $t > 0$, 
		consider an affine feedback controller with gain~$K=1$ and a stop price \mbox{$S_*< S_0$}.
		Then, the CDF for the cumulative trading profit or loss $F(x,t) = P(g(t) \le x)$ is described as follows: 
		For $x <  g_*,$ $F(x,t ) \equiv 0$, and, for~$x \ge g_*$,
		\begin{align*}
		F(x,t) & =1 - \Phi \left( {{X}\left( {  {\frac{u_0}{x + u_0}} ,t} \right)} \right)  + Z_*\Phi \left( {{X}\left( {  {{{\left( {\frac{{{S_*}}}{{{S_0}}}} \right)}^2}\frac{{{u_0}}}{{x + {u_0}}}} ,t} \right)} \right).
		\end{align*}
	\end{theorem}

	\begin{proof} Fix $K=1$. By Lemma~\ref{lemma: g_t for stopped GBM},
		the stopped GBM $\widetilde{S}(t)$ leads to 
		the cumulative trading profit or loss
		\[
		g(t) = u_0 \left( \frac{\widetilde{S}(t)}{S_0} - 1\right),
		\] 
		and we note that $g(t) \geq g_*$ for all $t$. Hence,  $g_*$ is the worst case, for~$x < g_*$, the cumulative distribution function~(CDF) is~$F(x,t) \equiv 0$. 
		Hence, we consider~$x \ge g_*$.
		Now, using a result in~\cite{Albanese_Camolieti_2005} for an arithmetic Brownian motion~(ABM)~$\mathcal{X}(t)$ with drift~$\alpha$ and volatility~$\sigma>0$ with stop~$\mathcal{X}_* < \mathcal{X}(0)$. 
		Namely, for~$x \ge \mathcal{X}_*$, the corresponding stopped ABM, call it~$\widetilde{\mathcal{X}}(t)$, has CDF
		\begin{align*}
		P({\widetilde{\mathcal{X}}(t)} \le x)  
		& =  1 - \Phi \left( {\frac{{\mathcal{X}\left( 0 \right) - x + \alpha t}}{{\sigma \sqrt t }}} \right)  + {e^{\frac{{2 \alpha }}{{{\sigma ^2}}}\left( {{\mathcal{X}_*} - \mathcal{X}\left( 0 \right)} \right)}}\Phi \left( {\frac{{2{\mathcal{X}_*} - \mathcal{X}\left( 0 \right) - x + \alpha t}}{{\sigma \sqrt t }}} \right).
		\end{align*}
		Now, returning to the buy and hold strategy, we apply Ito's Rule to transform from ABM to GBM by using $\mathcal{X}(t) = \log S(t)$ and $\alpha = \mu - \frac{1}{2}\sigma^2$. That is,
		Noting that $g(t) \le x$ corresponds to 
		$$
		\widetilde{S}(t) \le S_0 \left( \frac{x + u_0}{u_0} \right),
		$$ 
		we obtain
		\begin{align*}
		F(x,t)
		&=P(g(t) \le x) \\[.5ex]
		& = P\left( \widetilde S(t) \le S_0\left( \frac{x + u_0}{u_0} \right) \right)\\[.5ex]
		& = P\left( \log \widetilde S(t) \le \log S_0\left( \frac{x + u_0}{u_0} \right) \right)\\[.5ex]
		& = 1 - \Phi \left( {\frac{{\log \left( {\frac{{{u_0}}}{{x + {u_0}}}} \right) + \left( {\mu  - \frac{1}{2}{\sigma ^2}} \right)t}}{{\sigma \sqrt t }}} \right) \\[.5ex]
		& \hspace{5mm} +{\left( {\frac{{{S_*}}}{{{S_0}}}} \right)^{\frac{{2\mu }}{{{\sigma ^2}}} - 1}}\Phi \left( {\frac{{\log \left( {{{\left( {\frac{{{S_*}}}{{{S_0}}}} \right)}^2}\frac{{{u_0}}}{{x + {u_0}}}} \right) + \left( {\mu  - \frac{1}{2}{\sigma ^2}} \right)t}}{{\sigma \sqrt t }}} \right)\\[.5ex]
		& =1 - \Phi \left( {{X}\left( {  {\frac{{{u_0}}}{{x + {u_0}}}} ,t} \right)} \right)  + Z_* \Phi \left( {{X}\left( {  {{{\left( {\frac{{{S_*}}}{{{S_0}}}} \right)}^2}\frac{{{u_0}}}{{x + {u_0}}}} ,t} \right)} \right). 
		\end{align*}
	\end{proof}
	
	\subsection{Remarks}\label{subsec: Remarks for K=1}
	Consider $K=1$, using the preliminary notations introduced in Section~\ref{subsec: Preliminaries}, we recover the corresponding CDF for trading profit or loss $g(t)$ without stop; i.e., for~$x \le -u_0$, $ F_0(x,t) \equiv 0, $ and, for $x> -u_0$,
	\begin{align*}
	{F_0}(x,t) 
	&  = 1 - \Phi \left( {{X}\left( {  {\frac{{{u_0}}}{{x + {u_0}}}} ,t} \right)} \right).
	\end{align*}
	Suppose the stop price $S_* \to 0$.  
	Then we have $g_* \to -u_0$ and for $x > g_*$, the CDF becomes
	\begin{align*}
	F\left( {x,t} \right) 
	& \to 1 - \Phi \left( {{X}\left( {  {\frac{{{u_0}}}{{x + {u_0}}}} ,t} \right)} \right) = F_0(x,t).
	\end{align*}
	Namely, for $x > -u_0$, the CDF $F(x,t)$ reduces to the formula for $F_0(x,t)$.
	
	\subsection{Illustrative Example for \texorpdfstring{$K=1$}{K=1}} \label{subsec: Example for K=1}
	To understand the effects of stop loss order on stock trading, we consider the following simple example: Let~$K=1$ and take initial stock price~$S_0 =1$, initial investment~$u_0 = 1$, time $t=1$, volatility~$\sigma=1$, drift~$\mu=1/2$, and stop price $S_*=1/2$. 
	Note that the ratio $\mu/\sigma^2 \le 1/2$, which implies a downward trending stock prices. Thus, the specified stop price $S_*$ plays a role as a protection for the trading loss.
	Figure~\ref{fig:buynhold_MonteCarlo_pdf} shows the associated CDF plot for the trading profit or loss. 
	As depicted in the figure, the two solid lines (blue and red colored lines) are for the theoretical CDFs $F_0(x,t)$ and $F(x,t)$, respectively; and the dotted-line and dashed line indicate the corresponding CDF plots which are generated via Monte-Carlo simulations.
	The worst case of trading profit or loss for the trade without stop order is $-u_0 = -1$, but for the trade with stop order, the worst case of trading or loss becomes~$ g_* =  -0.5$. 
	The associated CDF at $x=g_*$ is $F(g_*,t) \approx 0.4882.$
	Furthermore, the CDF  $F(x,t) $ coincides with $F_0(x,t)$ when $x$ is sufficiently large. Said another way, this indicates that there is no difference between $F_0(x,t)$ and~$F(x,t)$ if the trade is winning; however, we should emphasize again that the trade with stop order gives a protection for losing trade. 
	The CDF of $g(t)$ in Figure~\ref{fig:buynhold_MonteCarlo_pdf} indeed matches perfectly with the one generated by using Monte-Carlo simulations using the~$50,000$ GBM sample paths and same parameters setting described above.
	
	
	%
	\begin{figure}
		\centering
		\includegraphics[scale=0.6]{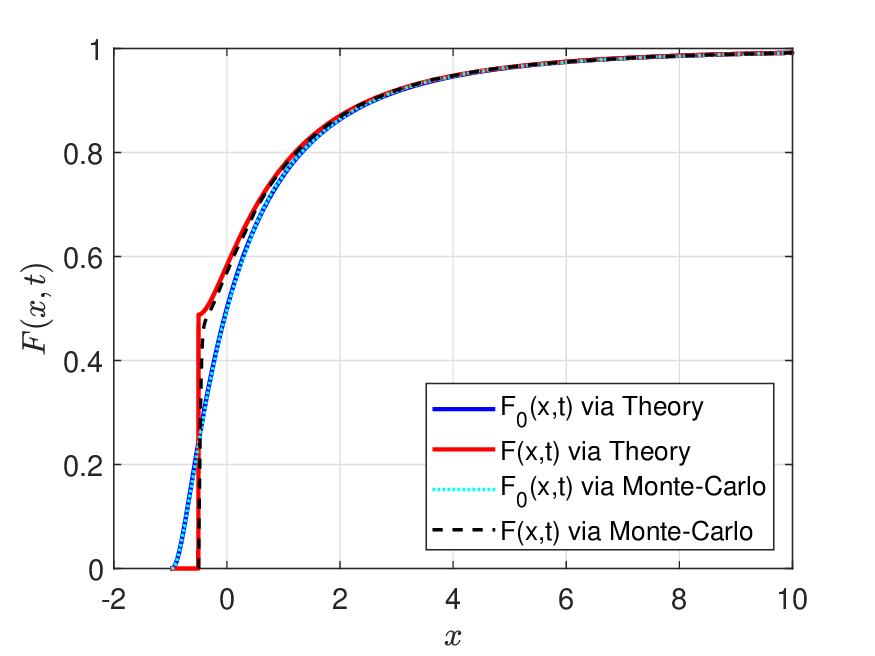}
		\caption{CDF plots for the bold trade ($K=1$) with and without stop-loss order.}
		\label{fig:buynhold_MonteCarlo_pdf}
	\end{figure}
	
	
	\vspace{5mm}
	\section{Bold Investment (\texorpdfstring{$K>1$}{K>1})}\label{SEC: CDF for Leverage Case}
	In this section, we now move to analyze the scenario for $K>1$, which we call the \textit{bold investment} case.
	
	\begin{theorem} \label{TheoremCDF2}
		For $t > 0$, 
		consider an affine feedback controller with gain \mbox{$K>1$} and a stop price~\mbox{$S_*< S_0$}.  Then, the CDF for the cumulative trading profit or loss $F(x,t)$  is described as follows:  For $x \le g_{*}(t)$, $F(x,t ) \equiv 0$,~for  
		\[
		g_*(t) < x < \frac{u_0}{K} \left[ {{{\left( {\frac{{{S_*}}}{{{S_0}}}} \right)}^K} - 1} \right],
		\]
		we have
		\begin{align*}
		F(x,t) 	&=   \Phi \left( {X\left( {\frac{{{S_0}}}{{{S_*}}},A(x)} \right)} \right) - {Z_*}\Phi \left( {X\left( {\frac{{{S_*}}}{{{S_0}}},A(x)} \right)} \right)\\[.5ex]
		&  \ \ \ \  - \Phi \left( {X\left( {\frac{1}{{B(x,t)}},t} \right)} \right)+ {Z_*}\Phi \left( {X\left( {{{\left( {\frac{{{S_*}}}{{{S_0}}}} \right)}^2}\frac{1}{{B(x,t)}},t} \right)} \right),
		\end{align*}
		and for 
		$
		x \ge \frac{{{u_0}}}{K}\left[ {{{\left( {\frac{{{S_*}}}{{{S_0}}}} \right)}^K} - 1} \right],
		$
		we get
		\begin{align*}
		F(x,t) 	&=	  1  - \Phi \left( {X\left( {\frac{1}{{B(x,t)}},t} \right)} \right)   + {Z_*}\Phi \left( {X\left( {{{\left( {\frac{{{S_*}}}{{{S_0}}}} \right)}^2}\frac{1}{{B(x,t)}},t} \right)} \right)  .
		\end{align*}
	\end{theorem}

	To prove Theorem \ref{TheoremCDF2}, some preliminary results, provided in the next subsection to follow, are needed.
	
	\subsection{Preliminaries} \label{subsec: Preliminaries}
	As mentioned in Section~\ref{SEC: Problem Formulation}, we recall that $t^*$ is the random stopping time satisfying
	$
	t^* := \inf \{t \geq 0: S(t) \le S_* \}.
	$ 
	For $S_0 > S_*$, it is well-known that the CDF for stopping time of GBM price has closed-form expression; e.g., see \cite{Albanese_Camolieti_2005}; namely, for~$t>0$,
	\begin{align}\label{eq: CDF for stopping time}
	P({t^*} \le t) = 1 - \Phi \left( {X\left( {\frac{{{S_0}}}{{{S_*}}},t} \right)} \right) + {Z_*}\Phi \left( {X\left( {\frac{{{S_*}}}{{{S_0}}},t} \right)} \right).
	\end{align}
	Furthermore, the joint probability for GBM and stopping time, with simple modification, is also discussed in \cite{Albanese_Camolieti_2005}; i.e., for $x,t>0$,
	\begin{align}\label{eq: joint prob for St and tstar}
	P(S(t) \ge x,{t^*} \ge t)&= \Phi \left( X\left( {\frac{S_0}{{x}},t} \right) \right)  - {Z_*}\Phi \left( {X\left( {\frac{{S_*^2}}{{{S_0}x}},t} \right)} \right).
	\end{align}
	As seen later in this paper, the two equations(\ref{eq: CDF for stopping time}) and (\ref{eq: joint prob for St and tstar}) play an important role in our derivation to follow.
	With the aids of these preliminaries, we first show Lemmas~\ref{lemma: Prob of g and stopped for K bigger than 1} and \ref{lemma: Prob of g and not yet stopped for K bigger than 1},  which are useful for proving Theorem~\ref{TheoremCDF2}.
	
	\begin{lemma} \label{lemma: Prob of g and stopped for K bigger than 1}
		For $t > 0$, 
		consider an affine feedback controller with gain $K>1$ and a stop price~\mbox{$S_*< S_0$}.  Then,
		\begin{align*}
		P\left( g(t) \le x , t^* \le t \right)  = 
		\begin{cases}
		1-\Omega(t), & {\rm if } \, x \geq \frac{u_0}{K}\left[ \left(\frac{S_*}{S_0}\right)^K-1 \right];\\
		\Omega(A(x))-\Omega(t), & {\rm if}\, g_*(t)<x<\frac{u_0}{K}\left[ \left(\frac{S_*}{S_0}\right)^K - 1\right];\\
		0, & {\rm if} \, x \leq g_*(t).
		\end{cases}
		\end{align*}
		where
		\begin{align*}
		& \Omega(t) =  \Phi \left( {X\left( {\frac{{{S_0}}}{{{S_*}}},t} \right)} \right) - {Z_*}\Phi \left( {X\left( {\frac{{{S_*}}}{{{S_0}}},t} \right)} \right);\\[.5ex]
		& \Omega(A(x)) =  \Phi \left( {X\left( {\frac{{{S_0}}}{{{S_*}}},A(x)} \right)} \right) - {Z_*}\Phi \left( {X\left( {\frac{{{S_*}}}{{{S_0}}},A(x)} \right)} \right).
		\end{align*}
	\end{lemma}
	
	\begin{proof}
		For $t^* \le t$, Lemma~\ref{lemma: g_t for stopped GBM} tells us that the corresponding trading gain is
		\begin{align}\label{eq:  g t_star}
		g(t) = g(t^*) ={\frac{{{u_0}}}{K}\left[ {{{\left( {\frac{{{S_*}}}{{{S_0}}}} \right)}^K}{e^{\frac{1}{2}{\sigma ^2}(K - {K^2}){t^*}}} - 1} \right]}.
		\end{align}
		Hence, we may write
		$P\left( {g(t) \le x, {t^*} \le t} \right) = P\left( {g\left( {{t^*}} \right) \le x , t^* \le t} \right).$
		In addition, since $t^*\leq t$, it follows that $g(t^*)\geq g_*(t)$ and $g_*(t)$ is the worst case. Thus, for $x < g_*(t)$, the the joint probability is
		$
		P\left( {g(t) \le x, {t^*} \le t} \right)  = 0.
		$
		
		On the other hand, we consider for $x \ge g_{*}(t)$. 
		Using equation~(\ref{eq:  g t_star}), we write the joint CDF formula as follows:
		\begin{align}\label{eq: Prob_g less than x and tstar less than t}
		P\left( {g(t) \le x ,{t^*} \le t} \right) = P\left( { {\frac{{{u_0}}}{K}\left[ {{{\left( {\frac{{{S_*}}}{{{S_0}}}} \right)}^K}{e^{\frac{1}{2}{\sigma ^2}(K - {K^2}){t^*}}} - 1} \right] \le x} , {t^*} \le t} \right).
		\end{align}
		Now, we break into two cases: For $x = g_*(t)$, with the aid of equation~(\ref{eq:g_star_t}), we have
		\begin{align*}
		P\left( {g(t) \le x ,{t^*} \le t} \right) 
		&= P\left( { {\frac{{{u_0}}}{K}\left[ {{{\left( {\frac{{{S_*}}}{{{S_0}}}} \right)}^K}{e^{\frac{1}{2}{\sigma ^2}(K - {K^2}){t^*}}} - 1} \right] \le g_*(t)} , {t^*} \le t} \right) \\[.5ex]
		& = P\left( {{t^*} \ge t,{t^*} \le t} \right) =0.
		\end{align*}
		On the other hand, for $x > g_*(t)$, we write the equation~(\ref{eq: Prob_g less than x and tstar less than t}) as follows:
		\begin{align*}
		P\left( {g(t) \le x ,{t^*} \le t} \right) & = P\left( { {{t^*} \ge \frac{{ - 2}}{{{\sigma ^2}K(K - 1)}}\log \left[ {{{\left( {\frac{{{S_0}}}{{{S_*}}}} \right)}^K}\left( {\frac{K}{{{u_0}}}x + 1} \right)} \right]} , {t^*} \le t} \right) 
		\end{align*}
		and it is readily verified that right-hand side involving logarithm is well-defined. Now, we define
		\begin{align*}
		P\left( {g(t) \le x ,{t^*} \le t} \right) 
		& := P\left( {{t^*} \ge A\left( x \right),  t^* \le t} \right)
		\end{align*}
		where
		\[
		A(x) = {\frac{{ -2}}{{{\sigma ^2}K(K - 1)}}\log \left[ {{{\left( {\frac{{{S_0}}}{{{S_*}}}} \right)}^K}\left( {\frac{K}{{{u_0}}}x + 1} \right)} \right]}.
		\]
		Note that the sign of $A(x)$ depends on the logarithm part; i.e., if the argument of logarithm function is within the interval $(0,1)$, then the associated function value is negative; however, if the argument of  logarithm function is greater than or equal to one, then the associated function value is non-negative. Therefore, we have two subcases to analyze in this $x  >  g_{*}(t)$ scenario:
		
		{\it Case 1:} For 
		\[ 
		{g_*}(t) < x < \frac{{{u_0}}}{K}\left[ {{{\left( {\frac{{{S_*}}}{{{S_0}}}} \right)}^K} - 1} \right],
		\] within such range of $x$, we get $A(x) > 0$. Thus, the joint CDF formula becomes
		\begin{align*}
		P\left( {g(t) \le x, {t^*} \le t} \right)
		&  = P({t^*} \ge A(x), {t^*} \le t)\\[.5ex]
		& = {{P\left( {A(x) \le {t^*} \le t} \right)}}\\[.5ex]
		& = {{P\left( {{t^*} \le t} \right) - P\left( {{t^*} \le A\left( x \right)} \right)}}\\[.5ex]
		&  =  - \Phi \left( {X\left( {\frac{{{S_0}}}{{{S_*}}},t} \right)} \right) + {Z_*}\Phi \left( {X\left( {\frac{{{S_*}}}{{{S_0}}},t} \right)} \right) \\[.5ex]
		& \ \ \ \ \ \ \ + \Phi \left( {X\left( {\frac{{{S_0}}}{{{S_*}}},A\left( x \right)} \right)} \right) - {Z_*}\Phi \left( {X\left( {\frac{{{S_*}}}{{{S_0}}},A\left( x \right)} \right)} \right)
		\end{align*} 
		where the last equality holds by applying the preliminary result of the CDF for the stopping time; i.e., equation~(\ref{eq: CDF for stopping time}).
		
		{\it Case 2:} For
		\[
		x \ge  \frac{{{u_0}}}{K}\left[ {{{\left( {\frac{{{S_*}}}{{{S_0}}}} \right)}^K} - 1} \right],
		\]within such range of $x$, the function $A(x) \le 0$; however, since the stopping time $t^*$ is non-negative, we get
		\begin{align*}
		P\left( {g(t) \le x , {t^*} \le t} \right) 
		&= P\left( {{t^*} \ge A\left( x \right) , t^* \le t} \right) \\[.5ex]
		& = P( t^* \le t)\\[.5ex]
		& =1 - \Phi \left( {X\left( {\frac{{{S_0}}}{{{S_*}}},t} \right)} \right)  + {Z_*}\Phi \left( {X\left( {\frac{{{S_*}}}{{{S_0}}},t} \right)} \right)
		\end{align*}
		where the last equality holds by applying equation~(\ref{eq: CDF for stopping time}) again. To sum up, we obtain
		\begin{align*}
		&P\left( g(t) \le x , t^* \le t \right) = \left\{ {\begin{array}{*{20}{l}}
			{1 - \Omega(t),\begin{array}{*{20}{c}}
				{}&{}&{}&{} \ \
				\end{array}{\rm{if}} \ x \ge \frac{{{u_0}}}{K}\left[ {{{\left( {\frac{{{S_*}}}{{{S_0}}}} \right)}^K} - 1} \right]}\\[.5ex]
			{{{ \Omega \left( {A\left( x \right)} \right) - \Omega \left( t \right)  }},\begin{array}{*{20}{c}}
				{}
				\end{array}{\rm{if}}\;g_{*}(t) < x < \frac{{{u_0}}}{K}\left[ {{{\left( {\frac{{{S_*}}}{{{S_0}}}} \right)}^K} - 1} \right]}\\[.5ex]
			{0,\begin{array}{*{20}{c}}
				{}&{}&{}&{}&{} &{} &{} \ \
				\end{array}{\rm{if}}\;x \le g_{*}(t)}
			\end{array}} \right.
		\end{align*}
		where
		\begin{align*}
		& \Omega(t) =  \Phi \left( {X\left( {\frac{{{S_0}}}{{{S_*}}},t} \right)} \right) - {Z_*}\Phi \left( {X\left( {\frac{{{S_*}}}{{{S_0}}},t} \right)} \right);\\[.5ex]
		& \Omega(A(x)) =  \Phi \left( {X\left( {\frac{{{S_0}}}{{{S_*}}},A(x)} \right)} \right) - {Z_*}\Phi \left( {X\left( {\frac{{{S_*}}}{{{S_0}}},A(x)} \right)} \right)
		\end{align*}
		and the proof is complete.
	\end{proof}
	
	\begin{lemma} \label{lemma: Prob of g and not yet stopped for K bigger than 1}
		For $t > 0$, 
		consider an affine feedback controller with gain $K>1$ and a stop price~\mbox{$S_*< S_0$}.  Then,
		\begin{align*}
		P\left( {g(t) \le x , {t^*} > t} \right)= \begin{cases}
		\Omega(t) - \Theta(x,t), & {\rm if } \, x> g_*(t);\\
		0, & {\rm if} \, x \leq g_*(t)
		\end{cases}
		\end{align*}
		where
		\[
		\Theta(x,t) := \Phi \left( X\left( {\frac{1}{B\left( {x,t} \right)},t} \right) \right) - Z_*\Phi \left( {X\left( {{{\left( {\frac{{{S_*}}}{{{S_0}}}} \right)}^2}\frac{1}{{B\left( {x,t} \right)}},t} \right)} \right)
		\] and $\Omega(t)$ is defined in Lemma~\ref{lemma: Prob of g and stopped for K bigger than 1}.
	\end{lemma}
	
	\begin{proof}
		Next, we move to compute $P(g(t) \le x , t^* > t)$. 
		For $t^* > t$, Lemma~\ref{lemma: g_t for stopped GBM} tells us that
		\[
		g(t) = \frac{{{u_0}}}{K}\left[ {{{\left( {\frac{{S(t)}}{{{S_0}}}} \right)}^K}{e^{\frac{1}{2}{\sigma ^2}(K - {K^2})t}}\; - 1} \right].
		\]
		Therefore,  the corresponding probability~\mbox{$P ( g(t) \le x , t^* > t )$} satisfies 
		\begin{align*}
		P\left( {g(t) \le x , {t^*} > t} \right)  = P\left( { {\frac{{{u_0}}}{K}\left[ {{{\left( {\frac{{S(t)}}{{{S_0}}}} \right)}^K}	{e^{\frac{1}{2}{\sigma ^2}(K - {K^2})t}} - 1} \right] \le x} , {t^*} > t} \right).
		\end{align*}
		For $x< g_{*}(t)$, $P\left( {g(t) \le x, {t^*} > t} \right) =0$. On the other hand, for $x \ge g_{*}(t)$, with the aid of equation~(\ref{eq: B_x_t}), we have
		\begin{align}\label{eq: prob S_t leq x and t_start less t}
		P\left( g(t) \le x, t^* > t \right) = P\left( S(t) \le S_0 B(x,t),{t^*} > t \right).
		\end{align} 
		Now, by the law of total probability; e.g., see~\cite{Gubner_2006}, we have
		\begin{align*}
		P(S\left( t \right) \le S_0 B(x,t)) &= P(S\left( t \right) \le S_0 B(x,t),{t^*} > t) + P(S\left( t \right) \le S_0 B(x,t),{t^*} \le t).
		\end{align*}
		It implies that the joint CDF 
		\begin{align*}
		P(S\left( t \right) \le S_0 B(x,t),{t^*} > t) &= P(S\left( t \right) \le S_0 B(x,t))  - P(S\left( t \right) \le S_0 B(x,t),{t^*} \le t).
		\end{align*}
		Note that the second term of the equality above can be determined by using the following fact: That is,
		given any two continuous random variables $(X,Y)$, it is readily verified that the identity 
		\begin{align*}
		P(X \ge x,Y \ge y) &+ P(X \le x) + P(Y \le y)   = 1 + P(X \le x,Y \le y) 
		\end{align*}
		holds.
		Thus, using equation~(\ref{eq: prob S_t leq x and t_start less t}) and applying the above fact, we write the joint CDF as follows
		\begin{align*}
		P\left( g(t) \le x,{t^*} > t \right)
		& = P(S\left( t \right) \le S_0 B(x,t),{t^*} > t) \\[.5ex]
		& = P(S\left( t \right) \le S_0 B(x,t))  - P(S\left( t \right) \le S_0 B(x,t),{t^*} \le t) \\[.5ex]
		& = 1 - P({t^*} \le t) - P(S\left( t \right) \ge S_0 B(x,t),{t^*} \ge t) \\[.5ex]
		& =   \Phi \left( {X\left( {\frac{{{S_0}}}{{{S_*}}},t} \right)} \right) - {Z_*}\Phi \left( {X\left( {\frac{{{S_*}}}{{{S_0}}},t} \right)} \right) \\[.5ex]
		& \ \ \ \ \ \  - \Phi \left( {X\left( {\frac{1}{{B\left( {x,t} \right)}},t} \right)} \right) + {Z_*}\Phi \left( {X\left( {{{\left( {\frac{{{S_*}}}{{{S_0}}}} \right)}^2}\frac{1}{{B\left( {x,t} \right)}},t} \right)} \right)
		\end{align*}
		where the last equality holds by applying the equations~(\ref{eq: CDF for stopping time}) and (\ref{eq: joint prob for St and tstar}).
		
		Thus, we obtain 
		\begin{align*}
		&P\left( {g(t) \le x , {t^*} > t} \right)= \left\{ \begin{array}{l}
		\Omega (t) - \Theta \left( {x,t} \right)\begin{array}{*{20}{c}}
		{}
		\end{array},{\rm{if}}\;x > {g_*}(t)\\[.5ex]
		0\begin{array}{*{20}{c}}
		{}&{}&{}&{}&{}&{}\;\;\;
		\end{array},{\rm{if}}\;x \le {g_*}(t).
		\end{array} \right.
		\end{align*}
		where
		\[
		\Theta(x,t) := \Phi \left( {X\left( {\frac{1}{{B\left( {x,t} \right)}},t} \right)} \right) - {Z_*}\Phi \left( {X\left( {{{\left( {\frac{{{S_*}}}{{{S_0}}}} \right)}^2}\frac{1}{{B\left( {x,t} \right)}},t} \right)} \right).
		\]
	\end{proof}
	
	Equipped with Lemmas~\ref{lemma: Prob of g and stopped for K bigger than 1} and \ref{lemma: Prob of g and not yet stopped for K bigger than 1}, we are now ready to give a proof for Theorem~\ref{TheoremCDF2}.
	
	\begin{proof}[Proof of Theorem \ref{TheoremCDF2}]
		We first define the minimum of GBM over $[0,t]$; i.e.,
		\[
		m(t) := \min_{u \in [0,t]}S(u).
		\]
		To compute the CDF for the trading gain~$F(x,t)$, we apply the law of total probability; e.g., see~\cite{Gubner_2006}, and write 
		\begin{align}\label{eq:F(x,t)}
		F(x,t) 	& = P(g(t) \le x, m(t) \le S_*) + P(g(t) \le x, m(t) > S_* ). 
		\end{align}
		Note that the events 
		$
		\{ m(t) \le S_*\} = \{ t^* \le t \}$ and $
		\{ m(t) > S_*\} = \{t^* > t \}.
		$
		Thus, we rewrite equation (\ref{eq:F(x,t)}) and obtain
		\begin{align*}
		F(x,t) 	& = P(g(t) \le x, t^* \le t) + P(g(t) \le x, t^* >t ) .
		\end{align*}
		According to Lemmas~\ref{lemma: Prob of g and stopped for K bigger than 1} and \ref{lemma: Prob of g and not yet stopped for K bigger than 1}, we obtain:
		For $x \le  g_*(t)$, $F(x,t ) \equiv 0$, and,~for 
		\[
		g_{*}(t) < x < \frac{{ {u_0}}}{K}\left[ { {{\left( {\frac{{{S_*}}}{{{S_0}}}} \right)}^K}} -1 \right],
		\]
		we have
		\begin{align*}
		F(x,t) 	&=   \Phi \left( {X\left( {\frac{{{S_0}}}{{{S_*}}},A(x)} \right)} \right) - {Z_*}\Phi \left( {X\left( {\frac{{{S_*}}}{{{S_0}}},A(x)} \right)} \right)\\[.5ex]
		&  \ \ \ \  - \Phi \left( {X\left( {\frac{1}{{B(x,t)}},t} \right)} \right)+ {Z_*}\Phi \left( {X\left( {{{\left( {\frac{{{S_*}}}{{{S_0}}}} \right)}^2}\frac{1}{{B(x,t)}},t} \right)} \right),
		\end{align*}
		and for 
		\[
		x \ge  \frac{{ {u_0}}}{K}\left[ { {{\left( {\frac{{{S_*}}}{{{S_0}}}} \right)}^K}} -1 \right],
		\] we get
		\begin{align*}
		F(x,t) 	&=	  1  - \Phi \left( {X\left( {\frac{1}{{B(x,t)}},t} \right)} \right)    + {Z_*}\Phi \left( {X\left( {{{\left( {\frac{{{S_*}}}{{{S_0}}}} \right)}^2}\frac{1}{{B(x,t)}},t} \right)} \right). 
		\end{align*}
		
	\end{proof}
	
	\subsection{Remarks}\label{subsec: Remarks for K>1} Let $K>1.$
	letting the stop price $S_* \to 0$, then for
	$$x \geq \frac{u_0}{K}\left[ \left( {\frac{{{S_*}}}{{{S_0}}}} \right)^K - 1 \right] \to -u_0/K,$$ 
	Theorem~\ref{TheoremCDF2} tells us that the CDF 
	\begin{align*}
	F(x,t) 	&=	  1  - \Phi \left( {X\left( {\frac{1}{{B(x,t)}},t} \right)} \right)   + {Z_*}\Phi \left( {X\left( {{{\left( {\frac{{{S_*}}}{{{S_0}}}} \right)}^2}\frac{1}{{B(x,t)}},t} \right)} \right)\\[.5ex]
	& \to 1  - \Phi \left( {X\left( {\frac{1}{{B(x,t)}},t} \right)} \right) = F_0(x,t).
	\end{align*}
	In other words, if the trade is not stopped, then Theorem~\ref{TheoremCDF2} replicates the CDF for~$g(t)$ without stop, which is obtained in Lemma~\ref{lemma: CDF for g(t) without stop order}.  
	
	\subsection{Illustrative Example for \texorpdfstring{$K>1$}{K>1}} 
	Similar to previous Example~\ref{subsec: Example for K=1}, to illustrate the theory, we consider the following simple example: Let~$K=2$ and take~$S_0 =u_0 =t=\sigma=1$, and $\mu=S_*=1/2$. 
	Figure~\ref{fig:cdfboldkbigger1montecarlo} reveals the associated CDF plots for the trading profit or loss $g(t)$. 
	The two solid lines are for the theoretical CDFs $F_0(x,t)$ and $F(x,t)$, and dotted-line and dashed line indicate the corresponding CDF plots which are generated via Monte-Carlo simulations.
	The worst case of trading profit or loss for the trade without stop is $-u_0/K = -1/2$, but for the trade with stop, the worst case becomes of trading profit or loss becomes
	$
	g_*(t) =  -0.3823.
	$
	Furthermore, consistent with the intuition, Figure~\ref{fig:cdfboldkbigger1montecarlo} also shows that the CDF $F(x,t)$ coincides with $F_0(x,t)$ for sufficiently large $x$. 
	Finally, Figure~\ref{fig:cdfboldkbigger1montecarlo} indicates that the theoretical CDF matches perfectly with that generated via Monte-Carlo simulations.
	
	
	\begin{figure}
		\centering
		\includegraphics[scale=0.6]{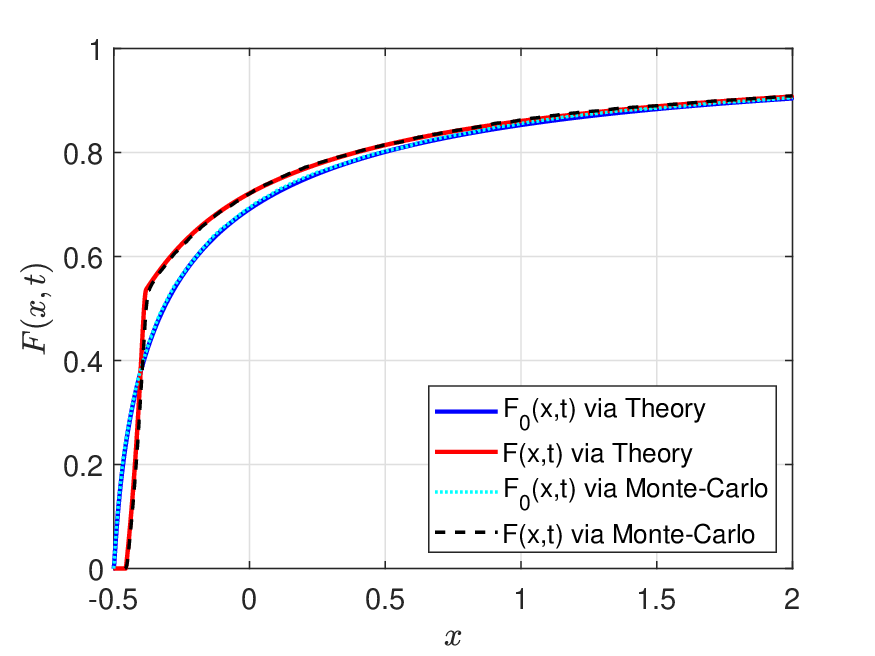}
		\caption{CDF plots for the bold trade ($K>1$) with and without stop-loss order.}
		\label{fig:cdfboldkbigger1montecarlo}
	\end{figure}

	\vspace{5mm}
	\section{Timid Investment
		(\texorpdfstring{$0<K<1$}{0<K<1})}\label{SEC: CDF for Cash-Financing}
	As mentioned in Sections~\ref{SEC: Problem Formulation} and \ref{Sec: Preliminary Notations}, the case where $0<K<1$ indeed closely related to the so-called {\it cash-financing} condition in finance literature; e.g., see \cite{Barmish_Primbs_2016,hsieh2019positive}. We also note that this topic is revisited in a greater detail in Section~\ref{SEC:Some Technical Results}. Having said this, we are now ready to state our next theorem which characterizes the CDF for $g(t)$ when $0<K<1.$

	\begin{theorem}\label{TheoremCDF3}
		For $t > 0$, consider
		an affine feedback controller with feedback gain $0<K<1$ and a stop price~$S_*< S_0$.  Then, the CDF for the trading profit or loss~$F(x,t)$  is described as follows:  For 
		$x \leq  \frac{u_0}{K}\left[ {{{\left( \frac{S_*}{S_0} \right)}^{K}} - 1} \right],$
		we have $F(x,t ) \equiv 0$, for $\frac{u_0}{K}\left[ {{{\left( \frac{S_*}{{{S}_{{0}}}} \right)}^{K}} - 1} \right] < x < g_*(t)$, we have 
		$
		F(x,t)=P(t^*\leq A(x)),
		$
		and, for  $x \ge g_*(t),$ we get
		\begin{align*}
		F(x,t) 	&=	  1  - \Phi \left( {X\left( {\frac{1}{{B(x,t)}},t} \right)} \right)   + {Z_*}\Phi \left( {X\left( \left( {\frac{{{S_*}}}{{{S_0}}}} \right)^2\frac{1}{{B(x,t)}},t \right)} \right)  .
		\end{align*}
	\end{theorem}

	Similarly, to prove Theorem~\ref{TheoremCDF3}, we need some aids from two preliminary lemmas; e.g., Lemmas~\ref{lemma: prob in cash-financing 1 } and \ref{lemma: prob in cash-financing 2 }. 
	Similar to Section~\ref{SEC: CDF for Leverage Case}, let $t^* = \inf \{t \geq 0: S(t) \le S_* \}$. We now give proofs for the two lemmas. After that, as seen later in the section, the proof of Theorem~\ref{TheoremCDF3} follows.

	\begin{lemma} \label{lemma: prob in cash-financing 1 }
		For $t > 0$, consider an affine feedback controller with gain $0<K<1$ and a stop price~$S_*< S_0$. Then 
		\[
		P\left( g(t) \leq x, t^* \leq t \right) = \begin{cases}
		P(t^*\leq t), &\text{ if } x \geq g_*(t);\\
		P(t^* \leq A(x)), &\text{ if }  \frac{u_0}{K}\left[ {{{\left( {\frac{{{S_*}}}{{{{S}_{{0}}}}}} \right)}^{K}} - 1} \right] < x < g_*(t);\\
		0, & \text{ if } x \leq \frac{u_0}{K}\left[ {{{\left( {\frac{{{S_*}}}{{{{S}_{{0}}}}}} \right)}^{K}} - 1} \right]
		\end{cases}
		\]
		where $P(t^*\leq z)$ for $z>0$ is given in equation~(\ref{eq: CDF for stopping time}).
	\end{lemma}

	\begin{proof}
		We follow a similar approach in the proof of Lemma~\ref{lemma: Prob of g and stopped for K bigger than 1}. We begin by noting that for $t^* \le t$, Lemma~\ref{lemma: g_t for stopped GBM} tells us that the trading profit or loss satisfies
		\begin{align}\label{eq: g_tstar formula}
		g(t) = g(t^*) =\frac{u_0}{K}\left[ {{{\left( {\frac{S_*}{S_0}} \right)}^K}{e^{\frac{1}{2}{\sigma ^2}(K - {K^2}){t^*}}} - 1} \right].
		\end{align}
		Thus, we  may write $
		P\left( g(t) \le x, t^* \le t \right) = P\left( g\left( {{t^*}} \right) \le x , t^* \le t \right).
		$
		Since $K \in (0,1)$, it is readily verified that
		\begin{align}\label{ineq: worst case}
		g_*(t) \ge g(t^*) \ge \frac{u_0}{K}\left[ \left( \frac{S_*}{S_0} \right)^K -1\right].
		\end{align}
		The inequality~(\ref{ineq: worst case}) above tells us that  $\frac{u_0}{K}\left[ {{\left( {\frac{{{S_*}}}{{{S_{0}}}}} \right)}^K} - 1 \right]$ is the worst case of trading profit or loss. Hence, this implies that for 
		$
		x < \frac{u_0}{K}\left[ {{{\left( \frac{S_*}{S_0} \right)}^K} - 1} \right],
		$ the joint probability $$P\left( {g(t) \le x, {t^*} \le t} \right)  = 0.$$
		On the other hand, for 
		$x \ge \frac{u_0}{K}\left[ {{{\left( \frac{S_*}{S_0} \right)}^{K}} - 1} \right],$
		since $K\in (0,1)$, using the equation~(\ref{eq: g_tstar formula}),  the joint CDF formula becomes
		\begin{align*}
		P\left( g(t) \le x , t^* \le t \right) 
		& = P(g(t^*)\leq x, t^* \leq t) \\[.5ex] 
		&= P\left( { {\frac{{{u_0}}}{K}\left[ {{{\left( {\frac{{{S_*}}}{{{S_0}}}} \right)}^K}{e^{\frac{1}{2}{\sigma ^2}(K - {K^2}){t^*}}} - 1} \right] \le x} , t^* \le t} \right)\\[.5ex]
		&  =P\left( t^* \le \frac{2}{{\sigma ^2}K(1 - K)}\log \left[ {{{\left( {\frac{S_0}{S_*}} \right)}^K}\left( {\frac{K}{{{u_0}}}x + 1} \right)} \right],{t^*} \le t \right).
		\end{align*}
		Using the definition of $A(x)$ given in equation~(\ref{eq: A_x}), we obtain
		\begin{align} \label{eq: Prob eq }
		P\left( g(t) \le x , t^* \le t \right) 
		&  = P\left( {{t^*} \le A\left( x \right),  t^* \le t} \right). 
		\end{align}
		Note that within the range of $x \ge \frac{u_0}{K}\left[ {{{\left( \frac{S_*}{S_0} \right)}^{K}} - 1} \right]$, the function $A(x) \geq 0$; and $A(x)=0$ if $x = \frac{u_0}{K}\left[ \left( \frac{S_*}{S_0} \right)^K - 1 \right]$. In this case, $P(t^*\leq A(x), t^*\leq t) =P(t^*\leq 0) = 0.$ 
		To complete the proof, below we proceed our analysis for two remaining cases: 
		
		\textit{Case 1:} For $x \geq g_*(t)$, we claim that the following event identity holds:
		\begin{align}\label{eq: events}
		\{t^* \leq A(x), t^* \leq t\} = \{t^* \leq t\}.
		\end{align}
		To see this, recalling that 
		$$ 
		g_*(t) = \frac{u_0}{K}\left( {{{\left( {\frac{{{S_*}}}{{{S_0}}}} \right)}^K}{e^{\frac{1}{2}{\sigma ^2}(K - {K^2})t}} - 1} \right),
		$$ 
		hence, $x\geq g_*(t)$ is equivalent to 
		$
		t \le 
		A(x)
		$
		and equation~(\ref{eq: events}) holds. Thus, equation~(\ref{eq: Prob eq }) becomes
		\begin{align*}
		P\left( g(t) \le x , t^* \le t \right) &= P\left( t^* \le A( x ) , t^* \le t \right) \\[.5ex]
		& = P(t^* \le t).
		\end{align*}
		
		\textit{Case 2:} For $ \frac{u_0}{K}\left[ {{{\left( {\frac{{{S_*}}}{{{{S}_{{0}}}}}} \right)}^{K}} - 1} \right] < x < g_*(t)$,  via a straightforward calculation, it follows that $t> A(x) > 0$. 
		Thus, equation~(\ref{eq: Prob eq }) becomes
		\begin{align*}
		P\left( {g(t) \le x , {t^*} \le t} \right) &= P\left( {{t^*} \le A\left( x \right) , t^* \le t} \right) \\[.5ex]
		& = P(t^* \le A(x)).
		\end{align*}
		Therefore, we obtain
		\[
		P\left( g(t) \leq x, t^* \leq t \right) = \begin{cases}
		P(t^*\leq t), &\text{ if } x \geq g_*(t);\\
		P(t^* \leq A(x)), &\text{ if }  \frac{u_0}{K}\left[ {{{\left( {\frac{{{S_*}}}{{{{S}_{{0}}}}}} \right)}^{K}} - 1} \right] < x < g_*(t);\\
		0, & \text{ if } x \leq \frac{u_0}{K}\left[ {{{\left( {\frac{{{S_*}}}{{{{S}_{{0}}}}}} \right)}^{K}} - 1} \right]
		\end{cases}
		\]
		and the proof is complete.
	\end{proof}

	\begin{lemma}\label{lemma: prob in cash-financing 2 }
		For $t > 0$, consider an affine feedback controller with gain $0<K<1$ and a stop price~$S_*< S_0$. Then for $x < g_*(t)$
		, then $P(g(t)\leq x, t^* > t) = 0$, and for $x \geq g_*(t)$, we have
		\begin{align*}
		P\left( {g(t) \le x, t^* > t} \right)
		&   =  1 - P({t^*} \le t)  - \Phi \left( {X\left( {\frac{1}{{B(x,t)}},t} \right)} \right)  \\[.5ex]
		& \hspace{10mm} + {Z_*}\Phi \left( {X\left( {{{\left( {\frac{{{S_*}}}{{{S_0}}}} \right)}^2}\frac{1}{{B(x,t)}},t} \right)} \right)
		\end{align*}
		where $P(t^*\leq t)$ is given in equation~(\ref{eq: CDF for stopping time}).
	\end{lemma}
	
	\begin{proof}
		For $t^* > t$, with the aid of Lemma~\ref{lemma: g_t for stopped GBM}, the  trading profit or loss $g(t)$ satisfies
		\begin{align*}
		g(t)  = \frac{{{u_0}}}{K}\left[ {{{\left( {\frac{{S(t)}}{{{S_0}}}} \right)}^K}{e^{\frac{1}{2}{\sigma ^2}(K - {K^2})t}}\; - 1} \right].
		\end{align*}
		Therefore,  the corresponding probability $P ( g(t) \le x , t^* > t )$ is
		\begin{align*}
		P\left( {g(t) \le x , {t^*} > t} \right) = P\left( { {\frac{{{u_0}}}{K}\left[ {{{\left( {\frac{{S(t)}}{{{S_0}}}} \right)}^K}	{e^{\frac{1}{2}{\sigma ^2}(K - {K^2})t}} - 1} \right] \le x} , {t^*} > t} \right).
		\end{align*}
		For $x < \frac{{{u_0}}}{K}\left[ {{{\left( {\frac{{{S_*}}}{{{{{S}}_{{0}}}}}} \right)}^{{K}}} - 1} \right],$
		we have 
		$$
		P\left( {g(t) \le x, {t^*} > t} \right) =0
		$$ and for
		$
		\frac{u_0}{K}\left[ {{{\left( {\frac{{{S_*}}}{{{{{S}}_{{0}}}}}} \right)}^{{K}}} - 1} \right] \leq x < g_*(t)
		$ 
		we have
		$
		P(g(t)\leq x, t^* > t)= 0
		$
		since $g(t) \geq g_*(t)$ for all $t<t^*.$
		Finally, for $x \geq g_*(t)$,
		we have
		$$
		P\left( {g(t) \le x, t^* > t} \right) = P\left( {S(t) \le S_0 B(x,t), t^*  > t} \right).
		$$ 
		Now, since 
		\begin{align*}
		P(S\left( t \right) \le S_0 B(x,t)) &= P(S\left( t \right) \le S_0 B(x,t),{t^*} > t) + P(S\left( t \right) \le S_0 B(x,t),{t^*} \le t),
		\end{align*}
		we can rewrite the joint CDF 
		\begin{align*}
		P(S\left( t \right) \le S_0 B(x,t),{t^*} > t) &= P(S\left( t \right) \le S_0 B(x,t)) - P(S\left( t \right) \le S_0 B(x,t),{t^*} \le t).
		\end{align*}
		Note that the second term of the equality above can be obtained by using the following facts:
		Given any two continuous random variables $(X,Y)$, we have the following identity
		$
		P(X \ge x,Y \ge y) + P(X \le x) + P(Y \le y) 
		= 1 + P(X \le x,Y \le y).
		$
		Thus, using the fact above, we write the joint CDF as follows:
		\begin{align*}
		P(S\left( t \right) \le S_0 B(x,t),{t^*} >t) 
		& = P(S\left( t \right) \le S_0 B(x,t))   - P(S\left( t \right) \le S_0 B(x,t),{t^*} \le t)\\[.5ex]
		&   = 1 - P({t^*} \le t) - P(S\left( t \right) \ge S_0 B(x,t),{t^*} \ge t) \\[.5ex]
		&   =  1 - P({t^*} \le t)  - \Phi \left( {X\left( {\frac{1}{{B(x,t)}},t} \right)} \right)  \\[.5ex]
		& \hspace{10mm} + {Z_*}\Phi \left( {X\left( {{{\left( {\frac{{{S_*}}}{{{S_0}}}} \right)}^2}\frac{1}{{B(x,t)}},t} \right)} \right)
		\end{align*}
		and the proof is complete.
	\end{proof}

	We are now ready to give our proof for Theorem~\ref{TheoremCDF3}.
	
	\begin{proof}[Proof of Theorem~\ref{TheoremCDF3}]
		The idea of the proof is similar to the case for $K>1$. However, for the sake of completeness, we provide a full proof here.
		Again, we begin by defining the minimum of GBM over~$[0,t]$ as
		$
		m(t) := \min_{u \in [0,t]}S(u).
		$
		To compute the CDF for the trading profit or loss~$F(x,t)$, we apply law of total probability and write 
		$
		F(x,t) 	 = P(g(t) \le x, m(t) \le S_*) + P(g(t) \le x, m(t) > S_* ). 
		$
		Note that the events
		$
		\{ m(t) \le S_*\} = \{ t^* \le t \}$
		and
		$
		\{ m(t) > S_*\} = \{t^* > t \}.
		$
		Thus, we have 
		$
		F(x,t) 	 = P(g(t) \le x, t^* \le t) + P(g(t) \le x, t^* >t ) .
		$
		According to Lemmas~\ref{lemma: prob in cash-financing 1 } and \ref{lemma: prob in cash-financing 2 }, the CDF for the cumulative trading profit or loss can be described as follows:
		For $x <  \frac{{{u_0}}}{K}\left[ {{{\left( {\frac{S_*}{S_0}} \right)}^{K}} - 1} \right],$
		we have $F(x,t )=0$, and, for $ \frac{{{u_0}}}{K}\left[ {{{\left( {\frac{{{S_*}}}{S_0}} \right)}^K} - 1} \right] \leq x < g_*(t)$,
		we have
		\begin{align*}
		F(x,t) &= 1 - P({t^*} \le t)  - \Phi \left( {X\left( {\frac{1}{{B(x,t)}},t} \right)} \right)   + {Z_*}\Phi \left( {X\left( {{{\left( {\frac{{{S_*}}}{{{S_0}}}} \right)}^2}\frac{1}{{B(x,t)}},t} \right)} \right).
		\end{align*}
		Finally, for $x \ge g_*(t),$ we have
		\begin{align*}
		F(x,t) 	&= 1  - \Phi \left( {X\left( {\frac{1}{{B(x,t)}},t} \right)} \right)   + {Z_*}\Phi \left( {X\left( {{{\left( {\frac{{{S_*}}}{{{S_0}}}} \right)}^2}\frac{1}{{B(x,t)}},t} \right)} \right)
		\end{align*}
		and the proof is complete.
	\end{proof}
	
	\subsection{Remark}
	Similar to Remarks~\ref{subsec: Remarks for K=1} and \ref{subsec: Remarks for K>1}, let $S_* \to 0$, then we have $F(x,t) \to F_0(x,t)$ for $x \ge -u_0/K$.

	\subsection{Illustrative Example for \texorpdfstring{$0<K<1$}{0<K<1}} 
	Similar to previous sections, to illustrate the theory, consider the following simple example: Let $K = 1/2$ and take parameters $S_0 =u_0 = t=\sigma=1$, and~$\mu=S_*=1/2$. 
	Figure~\ref{fig:cdftimidk01montecarlo} reveals the associated CDF plots for the trading profit or loss $g(t)$ where the two solid lines for the theoretical CDFs $F_0(x,t)$ and $F(x,t)$ and dotted-line and dashed line indicate the simulation counterpart which are generated via Monte-Carlo simulations.
	The worst case of trading profit or loss for the trade without stop is $ -u_0/K = -2$, but for the trade with stop, the worst case  is
	$
	\frac{u_0}{K}\left[ \left( \frac{S_*}{S_0} \right)^K - 1 \right] =  -0.5858.
	$
	Furthermore, Figure~\ref{fig:cdftimidk01montecarlo} shows that the CDF $F(x,t) $ coincides to $F_0(x,t)$ for sufficiently large $x$.
	Finally, the figure also indicates that the theoretical CDF plots match perfectly with those which are generated via Monte-Carlo simulations.
	

	\begin{figure}
		\centering
		\includegraphics[scale=0.6]{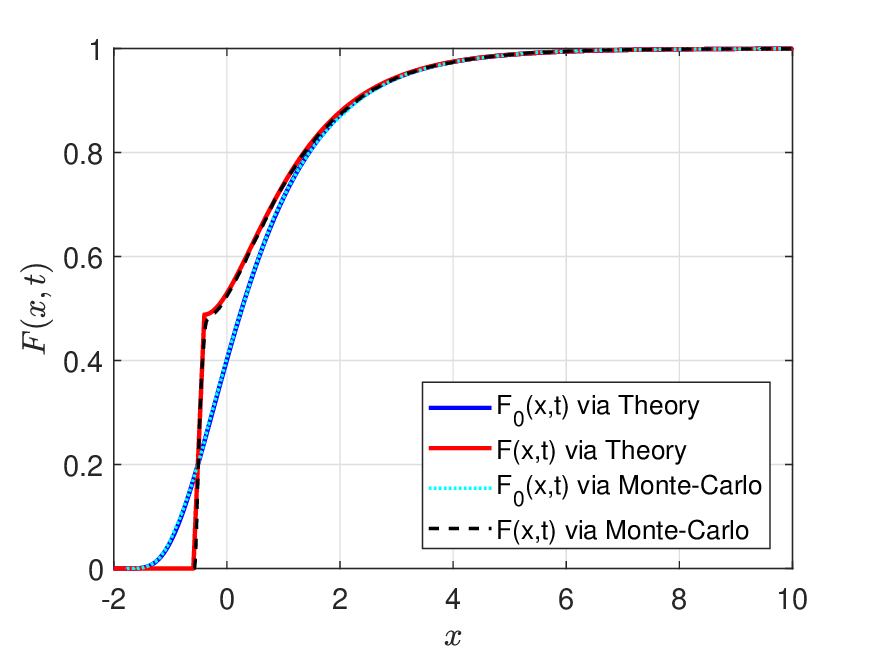}
		\caption{CDF plots for the timid trade ($0<K<1$) with and without stop-loss order.}
		\label{fig:cdftimidk01montecarlo}
	\end{figure}

	\vspace{5mm}
	\section{Some Technical Results}
	\label{SEC:Some Technical Results}
	In this section, we provide some additional technical results regarding the properties of affine feedback control that used in this paper. 
	Below, we first show that if one chooses $u_0 \leq KV_0$, then the so-called \textit{all-time survivability} is assured; i.e., $V(t) \geq 0$ for all $t$ with probability one. Said another way, ``no bankruptcy" is guaranteed; see also~\cite{hsieh2019positive}.
	
	\begin{lemma}[Survivability]\label{lemma: survivability}
		If $u_0 \leq KV_0$, then the trader's account value satisfies $
		V(t) \geq 0
		$ for all $t \geq 0$ with probability~one. 
	\end{lemma}
	
	\begin{proof}
		Let $t\geq 0$ be given. Along any GBM price path, we first note that the account value
		$V(t) = g(t)+ V_0.$ In addition, if $t < t_*$, it is readily verified that $	g(t) \geq g_*(t)$.
		Thus, using $u_0 \leq KV_0$, it follows that $V_0 \geq u_0/K$ and
		\begin{align*} \label{ineq: 1}
		V(t) & \geq V_0 +g_*(t) \\
		& \geq \frac{u_0}{K}\left(  \left( \frac{S_*}{S_0} \right)^K e^{\frac{1}{2}{\sigma ^2}(K - K^2)t}  \right) \\
		&> 0.
		\end{align*}
		On the other hand,	if $t \geq t^*$, then
		$
		g(t) = g(t^*) \geq -u_0/K.
		$
		Thus, $V(t) \geq V_0 - u_0/K \geq 0$ since $u_0 \leq KV_0.$
		Hence, the proof is complete.
	\end{proof}
	
	\subsection{Remarks} 
	The survivability issue is indeed closely related to the existing positive system theory. We refer the reader to our prior work in \cite{hsieh2019impact} and \cite{hsieh2019positive} for discussions on this topic.  
	It is also readily verified that 
	$
	|u(t)| \leq u_0 + KV_0 +KV(t)
	$
	for all $t>0$ with probability one. 
	To see this, using the facts that $u(t)=u_0 + Kg(t)$ and $g(t) = V(t)-V_0$, and Lemma~\ref{lemma: survivability}, we 
	have
	\begin{align*}
	|u(t)| 
	& \leq |u_0-KV_0| +KV(t)\\
	& \leq u_0 + KV_0 +KV(t)
	\end{align*}
	where the last inequalities above hold by using the triangle inequality. Moreover, with $u_0 \leq KV_0$, we obtain
	$
	|u(t)| \leq K(2V_0 +V(t)).
	$ 
	If $u_0=KV_0$, we can obtain a somewhat stronger result than Lemma~\ref{lemma: survivability}.
	
	\begin{corollary}\label{cor: control inputs} Let $K > 0$.\ \\
		(a) $u_0 = KV_0$ if and only if $u(t)=KV(t).$ \\
		(b) If $u_0=KV_0$, then $V(t) \geq 0$ and $u(t)\geq 0$ for all $t \geq 0$ with probability one.
	\end{corollary}
	
	\begin{proof} The proof of part (a) is straightforward. Indeed, note that 
		$$u(t)=u_0+Kg(t) = u_0-KV_0 +KV(t).$$ Hence, $u_0=KV_0$ implies that $u(t)=KV(t)$. Next, we assume $u(t) = KV(t)$. Then, by equating $u_0+Kg(t) = KV(t)$, we obtain that $u_0 = K(V(t)-g(t)) = KV_0.$
		
		To prove part~(b), assuming $u_0=KV_0$ and using Lemma~\ref{lemma: survivability}, it follows that $V(t)\geq 0$ with probability one. To see that $u(t) \geq 0$, we note that by part (a), it follows that
		$
		u(t) = KV(t) \geq 0
		$ where the last inequality holds since $K>0.$
	\end{proof}
	
	\subsection{Remark}\label{remark: cash-financing}
	Let $u_0 = KV_0$. If $K \in [0,1]$, then part (a) of Corollary~\ref{cor: control inputs} indicates that the trade is cash-financed. If $K >1$, then the trade is leveraged.

	
	%

	\begin{lemma}[Long Only Property]\label{lemma: long only property}
		If $K > 0$, then the affine feedback controller $u(t) > 0$ for all $t \in (0,t^*)$ with probability one, and $u(t)=0$ for all $t \geq t^*$.
	\end{lemma}
	
	\begin{proof}
		Since $K > 0$, $u_0>0$ and $g(t) \geq g_*(t)$ for all $t < t^*$, we have 
		\begin{align*}
		u(t) 
		&\geq u_0 + K g_*(t) \\
		& = u_0 + u_0\left( {{\left( \frac{S_*}{S_0} \right)^K}{e^{\frac{1}{2}{\sigma ^2}(K - K^2)t}} - 1} \right)\\
		&=u_0\left[  {{\left( \frac{S_*}{S_0} \right)^K}{e^{\frac{1}{2}{\sigma ^2}(K - K^2)t}} } \right]\\
		&>0.
		\end{align*}
		To complete the proof, we note that by definition, for $t \geq t^*$, $u(t)=0.$
	\end{proof}
	
	Subsequently, we provide a result which characterizes a lower bound of the expected cumulative trading profit or loss.
	
	\begin{lemma}[Lower Bounds for Expected Cumulative trading profit or loss] Let $u_0 \leq KV_0$. Then
		the expected cumulative trading profit or loss is lower bounded. Specifically, for constant $c>0$, we have
		$
		\mathbb{E}[g(t)] \geq c(1-F(c-V_0,t))-V_0
		$
		where $F(z,t)=P(g(t)\leq z)$ is the CDF for the cumulative trading profit or loss.
	\end{lemma}
	
	\begin{proof}We begin by recalling that $g(t)=V(t)-V_0$.  Since $u_0\leq KV_0$, Lemma~\ref{lemma: survivability} tells us that $V(t)\geq 0.$ Thus, applying the Markov inequality, we obtain, for $c>0$, 
		\begin{align*}
		P(V(t)>c) 
		&\leq 
		\frac{\mathbb{E}[g(t))]+V_0}{c}.
		\end{align*}
		This implies that
		\begin{align*}
		\mathbb{E}[g(t)]& \geq cP(V(t)>c)-V_0\\[.5ex]
		&=cP(g(t)>c-V_0)-V_0\\[.5ex]
		&=c(1-F(c-V_0,t))-V_0
		\end{align*}
		where the $F(c-V_0,t)=P(g(t)\leq c-V_0)$ is the CDF for $g(t)$ which obtained in previous sections. Hence, the proof is complete.
	\end{proof}
	
	
	\vspace{5mm}
	\section{Concluding Remarks and Future Work}
	\label{SEC:conclusions}
	In this paper,
	we provided a complete characterization for the CDF of the cumulative trading profit or loss function~$g(t)$ under a class of an affine feedback control with stop-loss order. 
	It should be noted that in this paper, we are silent about how to select the feedback gain $K$. Thus, an immediate research direction would be to consider the problem of finding an optimal feedback gain $K$ which maximizes some utility. For example, an optimal stopping problem which maximizes some expected reward objective would be an interesting direction to pursue; see also~\cite{zhang2001stock} for an initial work along this line.
	A second example would be to consider the celebrated  Kelly criterion, see~\cite{kelly2011new,thorp2011kelly}, which calls for an maximization of the expected logarithmic growth of one's account; i.e.,
	one may seek to find an optimal $K^*$ maximizing
	$
	\mathcal{G}(K):= \mathbb{E}\left[\log\frac{V(t)}{V_0}\right].
	$
	As a third example, one might consider other \textit{utility} function such as Markowitz's mean-variance criterion as performance metric; e.g., see \cite{markowitz1959}.  
	
	Another interesting future research direction would be to study the drawdown effect when the stop order is included in the control. Some initial researches along this line can be found in \cite{hsieh2017drawdown} and \cite{Hsieh_Barmish_2017_domination}. We should also note that in this paper, only one GBM stock is considered. Hence, as a natural extension, it would be to consider the case where multi-stock, or so-called portfolio rebalancing, is involved.
	

	\bibliographystyle{siamplain}
	
	\bibliography{Affine_Control_w_StopOrder_submitted}
\end{document}